\documentclass[12pt,reqno]{amsart}

\usepackage{eucal,color,graphicx,subfigure,mathrsfs}

\numberwithin{equation}{section}

\newtheorem{theorem}{Theorem}[section]
\newtheorem{lemma}[theorem]{Lemma}
\newtheorem{corollary}[theorem]{Corollary}

\theoremstyle{definition}

\theoremstyle{remark}
\newtheorem{remark}[theorem]{Remark}

\numberwithin{equation}{section}

\newcommand{\splitinf}[2]{\ensuremath{\inf_{\overset{\scriptstyle{#1}}{\scriptstyle{#2}}} }}

\newcommand{\diam}{\mathop{\mathrm{diam}}}

\newcommand{\tr}{\triangle}

\newcommand{\skw}{\mathop{\mathrm{skw}}}
\newcommand{\veps}{\varepsilon}

\def\d{\partial}

\newcommand{\Brg}{B_r^{\mathrm{grad}}}
\newcommand{\vpi}{\varPi}
\newcommand{\vpig}{\varPi_r^\mathrm{grad}}
\newcommand{\vpid}{\varPi_{p+2}^\mathrm{div}}
\newcommand{\vpidS}{\varPi_{p+2}^{(\mathrm{div},\mathbb{S})}}
\newcommand{\cpi}{C_{\scriptscriptstyle{\varPi}}}

\newcommand{\uu}{{\scriptstyle{\mathscr{U}}}}

\newcommand{\uuh}{\uu_{\!h}}
\newcommand{\vv}{{\scriptstyle{\mathscr{V}}}}
\newcommand{\vvv}{{\scriptscriptstyle{\mathscr{V}}}}
\newcommand{\ww}{{\scriptstyle{\mathscr{W}}}}
\newcommand{\www}{{\scriptscriptstyle{\mathscr{W}}}}
\newcommand{\bb}{{\scriptstyle{\mathscr{B}}}}
\newcommand{\xx}{{\scriptstyle{\mathscr{X}}}}

\newcommand{\uuhr}{\uu_{\!h}^{r}}
\newcommand{\Tr}{\smash[t]{T^{r}}}
\newcommand{\Vr}{\smash[t]{V^{r}}}
\newcommand{\Vhr}{\smash[t]{V_h^{r}}}

\newcommand{\gpi}{\varPi_{\mathrm{grad}}}
\newcommand{\dpi}{\varPi_{\mathrm{div}}}

\newcommand{\Frechet}{Fr{\'{e}}chet}
\newcommand{\Babuska}{Babu{\v{s}}ka}      
                 
\newcommand{\Poincare}{Poincar{\'{e}}}

\newcommand{\Epgrad}{E^{\grad}_p}
\newcommand{\Epdiv}{E^{\dive}_p}

\newcommand{\Hdiv}[1]{H(\mathrm{div},#1)}

\newcommand{\ip}[1]{\langle {#1} \rangle}

\newcommand{\om}{\varOmega}
\newcommand{\oh}{{\varOmega_h}}

\newcommand{\vx}{{\vec{x}}}

\newcommand{\sgn}{\hat{\sigma}_n}
\newcommand{\sgnh}{\hat{\sigma}_{n,h}}

\newcommand{\dive}{\ensuremath{\mathop{\mathrm{div}}}}
\newcommand{\grad}{\ensuremath{\mathop{\mathrm{grad}}}}

\newcommand{\VVV}{\mathbb{V}}
\newcommand{\RRR}{\mathbb{R}}

\newcommand{\KKK}{\mathbb{K}}

\begin{document}

\title{An analysis of the practical DPG method}

\author{J.~Gopalakrishnan}
\address{Portland State University,  PO Box 751, Portland, OR 97207-0751} 
\email{gjay@pdx.edu}

\author{W.~Qiu}
\address{The Institute for Mathematics and its Applications, 
University of Minnesota, Minneapolis, Minnesota 55455}
\email{qiuxa001@ima.umn.edu}
\thanks{Corresponding author: Weifeng Qiu (qiuxa001@ima.umn.edu)}

\thanks{This work was partly supported by the NSF under grants DMS-1211635 and
  DMS-1014817. The authors gratefully acknowledge
  the collaboration opportunities
  provided by the IMA (Minneapolis) during their 2010-11 program}

\begin{abstract}
  We give a complete error analysis of the Discontinuous
  Petrov Galerkin (DPG) method, accounting for all the approximations
  made in its practical implementation.  Specifically, we consider the
  DPG method that uses a trial space consisting of polynomials of
  degree $p$ on each mesh element. Earlier works showed that there is
  a ``trial-to-test'' operator $T$, which when applied to the trial
  space, defines a test space that guarantees stability. In DPG
  formulations, this operator $T$ is local: it can be applied
  element-by-element. However, an infinite dimensional problem on each
  mesh element needed to be solved to apply $T$. In practical
  computations, $T$ is approximated using polynomials of some degree
  $r > p$ on each mesh element. We show that this approximation
  maintains optimal convergence rates, provided that $r\ge p+N$, where
  $N$ is the space dimension (two or more), for the Laplace
  equation. We also prove a similar result for the DPG method for
  linear elasticity. Remarks on the conditioning of the stiffness
  matrix in DPG methods are also included.
\end{abstract}



\maketitle

\section{Introduction}

In this paper we prove error estimates for the discontinuous
Petrov-Galerkin (DPG) method applied to the Laplace equation and the equations of
linear elasticity. The approach is applicable more generally to other
equations as well. An error analysis of an ``ideal'' DPG method was
provided in~\cite{DemkoGopal:DPGanl}. Although the ideal method is not
practically implementable, a number of important theoretical tools for
analysis were developed in~\cite{DemkoGopal:DPGanl}.  We extend this
analysis using a few new lemmas to provide a complete analysis of the
fully implementable ``practical'' DPG method. The distinction between
the ideal and practical methods will be clear in the next few
paragraphs.

Both methods are easy to describe in a general context. Suppose we
want to approximate $\uu \in U$ satisfying
\begin{equation}
  \label{eq:1}
  b(\uu,\vv) = l(\vv), \quad \forall \vv \in V.
\end{equation}
Here $U$ is a Hilbert space with norm $\| \cdot\|_U$ and
$V$ is a Hilbert space under an inner product $(\cdot,\cdot)_V$ with 
corresponding norm $\| \cdot \|_V$. (All spaces are over $\RRR$.)  We
assume that the bilinear form $b(\cdot,\cdot): U \times V \mapsto
\RRR$ is continuous and the linear form $l(\cdot): V\mapsto \RRR$ is
also continuous. Define $T: U \mapsto V$ by
\begin{equation}
  \label{eq:T}
  (T\ww,\vv)_V = b(\ww,\vv),\qquad \forall  \vv \in V.
\end{equation}
Then, the DPG approximation to $\uu$, lies in a finite dimensional
trial subspace $U_h \subset U$ (where $h$ denotes a parameter
determining the finite dimension). It satisfies
\begin{equation}
  \label{eq:2}
  b(\uuh,\vv) = l(\vv),\quad\forall \vv \in V_h,
\end{equation}
where $V_h = T(U_h)$.  Since $U_h \ne V_h$ in general, this is a
Petrov-Galerkin approximation. The method~\eqref{eq:2} is the {\em
  ideal DPG method}. The excellent stability and approximation
properties of this method are well
known~\cite{DemkoGopal:2010:DPG2,DemkoGopal:DPGanl}.

The main difficulty of the ideal method is that in order to compute
$\uuh$, one needs a basis for $V_h$, which must be obtained by
applying $T$. This is infeasible, as seen from~\eqref{eq:T}, if $V$ is
infinite dimensional, unless a solution to~\eqref{eq:T} can be written
out in closed form. In certain one-dimensional problems, and in some
multi-dimensional problems like the transport equation, the
application of $T$ can be exactly written out in closed form
(see~\cite{DemkoGopal:2010:DPG1,DemkoGopal:2010:DPG2}). But for the
vast majority of interesting problems, this is not possible.

Yet, one may approximate $T$ by $\Tr$, defined as follows.  Let $\Vr$
be a finite dimensional subspace of $V$ (where $r$ is a parameter
determining the finite dimension). Then $\Tr \ww$ in $\Vr$ is defined
by
\begin{equation}
  \label{eq:Tr}
  (\Tr\ww,\vv)_V = b(\ww,\vv),\qquad \forall  \vv \in \Vr.
\end{equation}
One can then reconsider the DPG method~\eqref{eq:2} with $\Vhr =
\Tr (U_h)$ in place of $V_h$, i.e.,
\begin{equation}
  \label{eq:pract}
    b(\uuhr,\vv) = l(\vv),\quad\forall \vv \in \Vhr.
\end{equation}
This yields an implementable method that is very generally
applicable. We refer to this method as the {\em practical DPG method}.

A serious difficulty still remains when these ideas are applied to
standard variational problems. Namely, one application of $\Tr$ requires
inverting a Gram matrix in the $V$-inner product. This is
prohibitively expensive for most standard variational formulations.
For instance, if $V$ is $H^1(\om)$, where $\om$ is the computational
domain, then inverting the Gram matrix is as expensive as solving the
Laplace's equation.

This difficulty can be overcome by hybridization, as shown in the
earlier DPG
papers~\cite{DemkoGopal:2010:DPG2,DemkoGopal:DPGanl}. Namely, given a
boundary value problem, introducing certain interelement fluxes and
traces as new unknowns, we can design an {\em ultraweak} well-posed
variational formulation involving a space $V$ that contains functions
{\em discontinuous} across mesh element interfaces.  This then implies
that the Gram matrix becomes block diagonal, with one block per mesh
element (since $\Vr$ may now be chosen to be a DG subspace).  The
application of $\Tr$ is thus reduced to an easy block diagonal
inversion, i.e., the action of the operator $\Tr$ is {\em local}.

Such an ultraweak variational formulation has been developed for the
Poisson equation in~\cite{DemkoGopal:DPGanl}, where its wellposedness
is also proved. We will heavily rely on such wellposedness results in
this paper. An ultraweak formulation for the linear elasticity system
is also available now~\cite{BramDemkoGopalQiu:DPGelas}.  Both these
works analyzed the ideal DPG method~\eqref{eq:2} for the respective
ultraweak formulations.  The aim of the present paper is to provide an
error analysis for the corresponding practical DPG
methods~\eqref{eq:pract}.

In the next section we will present an approach to the analysis of the
practical method, continuing in the general context and using the
abstract notations introduced above. In Section~\ref{section_Laplace},
we perform the error analysis for the practical DPG method for the
Laplace equation. We also provide a condition number estimate. In
Section~\ref{sec:elas}, we consider a second example of linear
elasticity and provide an  error analysis.

\section{The approach to analysis} \label{sec:abstract}

The purpose of this section is to point out a simple functional analytic
route to proving the discrete stability of the practical DPG
method~\eqref{eq:pract}.  This discrete stability will follow from a
discrete inf-sup condition on the space~$\Vhr$. However, in
applications, it is often inconvenient to work directly with this
space. We prefer to work with $\Vr$, which will be some standard
polynomial space in most applications (as seen in the examples
later). The next theorem shows that the existence of a Fortin operator
into this more standard space $\Vr$ is a sufficient condition for the
convergence of the practical DPG method.

Before we give the result, let us state the assumptions that we shall
verify for each of our examples. We assume that
  \begin{equation}
    \label{eq:inj2}
    \{ \ww \in U :  \; b(\ww,\vv)  = 0,\; \forall \vv \in V \} = \{ 0 \}
  \end{equation}
   and that there is a positive constant $C_1$  such that 
   \begin{equation}
   \label{eq:equiv2}
    C_1 \| \vv \|_V  \le \sup_{\www\in U} \frac{ b(\ww,\vv) }{ \| \ww \|_U},
    \qquad \forall \vv \in V.
  \end{equation}
Above and throughout, we will tacitly assume that the suprema such as the above are taken over nonzero functions.
Let $C_2 \ge 0$ be such that
\begin{equation}
  \label{eq:b-cnt}
  b(\ww,\vv) \le C_2 \| \ww \|_U \| \vv \|_V,\qquad
  \forall \ww \in U, \; \vv \in V.
\end{equation}
Clearly, such a $C_2$ exists due to the continuity of
$b(\cdot,\cdot)$.  Finally, assume that there exists a linear operator
$\vpi: V \mapsto \Vr$ such that  for all $\vv \in V$, we have
\begin{subequations}
    \label{eq:pi}
    \begin{gather}
      \label{eq:pi-iden}
      b(\ww,\vv - \vpi \vv)
      =0, \qquad\forall \ww \in U_h, 
      \\
      \label{eq:pi-est}
      \| \vpi \vv \|_V \le \cpi \| \vv\|_V.
   \end{gather}
\end{subequations}

\begin{theorem}
  \label{thm:approach}
  Suppose the assumptions~\eqref{eq:inj2}, \eqref{eq:equiv2},
  \eqref{eq:b-cnt}, and~\eqref{eq:pi} hold.  Then the
  problem~\eqref{eq:1} is well-posed and
  \begin{equation}
  \label{eq:quasiopt}
  \| \uu  - \uuhr \|_U 
  \le \,\frac{C_2 \cpi }{C_1 } \,\inf_{\www \in U_h } \| \uu - \ww \|_U.
  \end{equation}
\end{theorem}
\begin{proof}
  We apply \Babuska's theory~\cite{Babus70,XuZikat03}. Accordingly, if
  we prove the discrete inf-sup condition
  \begin{equation}
    \label{eq:inf-sup-discrete}
    \frac{C_1}{\cpi} 
    \| \ww \|_U  \le \sup_{\vvv\in \Vhr } \frac{ b(\ww,\vv) }{ \| \vv \|_V},
    \qquad \forall \ww \in  U_h,
  \end{equation}
  then~\eqref{eq:quasiopt} will follow. We
  prove~\eqref{eq:inf-sup-discrete} in three steps, the first two of
  which are fairly standard (but included for readability).

As the first step, we prove that the following inf-sup condition holds:
  \begin{equation}
  \label{eq:infsup}
  C_1 \| \ww \|_U \le 
  \sup_{ \vvv\in V}\dfrac{\vert b(\ww,\vv)\vert}{\Vert \vv \Vert_{V}},
  \qquad
  \forall \ww\in U.
  \end{equation}
  This follows from the other inf-sup condition~\eqref{eq:equiv2}.
  Define a linear operator $B:U\rightarrow V^{*}$ by $
  B\ww=b(\ww,\cdot)\in V^{*},$ for all $\ww\in U.$ It is well
  known~\cite{Yosid95} that (\ref{eq:equiv2}) holds if and only if
  $B^{*}$ is injective and the range of $B^{*}$ is closed in~$U^{*}$.
  Additionally, by (\ref{eq:inj2}), $B$ is injective. Therefore, by
  the Closed Range Theorem, $B^{*}(V)=U^{*}$, so
  $(B^{*})^{-1}:U^{*}\rightarrow V$ exists.  Hence
  $B^{-1}:V^{*}\rightarrow U$ also exists and is continuous. This
  proves that problem~\eqref{eq:1} is {\em well-posed.}
  We obviously also have $\Vert B^{-1}\Vert=\Vert
  (B^{-1})^{*}\Vert=\Vert (B^{*})^{-1}\Vert$, i.e.,
  \begin{equation*}
    \inf_{\www\in U}\sup_{ \vvv\in V}\dfrac{\vert b(\ww,\vv)\vert}{\Vert \ww\Vert_{U}\Vert \vv\Vert_{V}}
    =\Vert B^{-1}\Vert^{-1}=\Vert (B^{*})^{-1}\Vert^{-1}
    =\inf_{ \vvv\in V}\sup_{ \www\in U}\dfrac{\vert b(\ww,\vv)\vert}{\Vert \ww\Vert_{U}\Vert \vv \Vert_{V}},
\end{equation*}
which proves~(\ref{eq:infsup}).

As the second step, we prove the following inf-sup condition.
  \begin{equation}
  \label{eq:infsupVr}
  \frac{C_1}{\cpi} \| \ww \|_U \le 
  \sup_{ \vvv\in \Vr }\dfrac{\vert b(\ww,\vv)\vert}{\Vert \vv \Vert_{V}},
  \qquad
  \forall \ww\in U_h.
  \end{equation}
  Note that this differs from~\eqref{eq:inf-sup-discrete} only in the
  space in which the supremum is sought. To prove~\eqref{eq:infsupVr},
  we use~\eqref{eq:infsup} and assumption~\eqref{eq:pi} as follows:
  \begin{align*}
    C_1 \| \ww \|_U 
    & \le \sup_{ \vvv \in V}  \frac{b(\ww,\vv) }{ \| \vv \|_V }
     \le  \sup_{ \vvv \in V}  \frac{b(\ww,\vpi \vv) }{ \cpi^{-1} \| \vpi \vv \|_V}.
  \end{align*}
  Now, since $\vpi\vv$ is in $\Vr$, the last supremum may be bounded
  by the supremum over all $\Vr$, so we obtain~\eqref{eq:infsupVr}.

As the  third and final step, we prove that
  if $s_1$ is the supremum in~\eqref{eq:infsupVr} and
  $s_2$ is the supremum in~\eqref{eq:inf-sup-discrete}, then 
  $s_1 = s_2$.  Obviously, $s_1 \ge s_2$ as $\Vr \supseteq
  \Vhr$. To prove the reverse inequality, observe that $ s_1 = \| \Tr
  \ww \|_V,$ by~\eqref{eq:Tr}. Since $\Tr \ww$ is in $\Vhr$, we have
  \[
  s_1 
  =\frac{ (\Tr \ww, \Tr \ww)_V }{ \| \Tr \ww \|_V } 
  \le  
  \sup_{\vv \in \Vhr } \frac{ ( \Tr \ww,\vv)_V }{ \| \vv \|_V } = s_2.
  \]
  Therefore, the inf-sup condition~\eqref{eq:inf-sup-discrete} follows
  from~\eqref{eq:infsupVr}.
\end{proof}

\begin{remark}[Test basis] \label{rem:basis} The above proof also
  shows that under the assumptions of Theorem~\ref{thm:approach}, the
  operator $\Tr : U_h \mapsto \Vr$ is {\em injective}: indeed, if $\Tr
  \ww =0$, then $b(\ww,\vv)=0$ for all $\vv$ in $\Vr$, so 
  by the inf-sup condition~\eqref{eq:infsupVr},
  we conclude that $\ww =0$. Note that the injectivity of $\Tr$ implies that 
  \[
  \dim (\Vhr) = \dim(U_h).
  \]
  It also implies that a basis for $\Vhr$ can be computed by applying
  $\Tr$ on any basis for $U_h$.
\end{remark}

\begin{remark}[Conditioning]      \label{rem:cond-abd}
  Suppose $\bb_i$ is a basis for $U_h$. Then, under the assumptions of
  Theorem~\ref{thm:approach}, $ \Tr \bb_i$ is a basis for $\Vhr$, as
  seen in Remark~\ref{rem:basis}. The $ij$-th entry of the {\em
    stiffness matrix} of the DPG method with respect to this basis is
  given by
  $
  S_{ij} = b( \bb_j,\Tr\bb_i) = (\Tr \bb_j, \Tr \bb_i)_V.
  $
  Clearly, $S$ is symmetric. The above mentioned injectivity of
  $\Tr$ implies that $S$ is also positive definite.  To understand the
  conditioning of~$S$, let us first note that
  \begin{equation}
    \label{eq:Tr-equiv}
    \frac{C_1}{\cpi} \| \ww \|_U \le \| \Tr \ww \|_V \le C_2 \| \ww \|_U,
    \qquad \forall \ww \in U_h. 
  \end{equation}
  This follows from the inf-sup condition~\eqref{eq:infsupVr} in the
  proof of Theorem~\ref{thm:approach} and the continuity
  property~\eqref{eq:b-cnt}.  Next, suppose
  $\xx = \sum_i x_i \bb_i$ is
  the basis expansion of any $\xx$ in $U_h$, and $\lambda_0,
  \lambda_1$ are positive numbers such that
  \begin{equation}
    \label{eq:lambdas}
    \lambda_0 \| x \|_{\ell^2}^2  \le \| \xx \|_U^2 \le \lambda_1 \| x \|_{\ell^2}^2, 
    \qquad\forall \xx \in U_h.
  \end{equation}
  Since $x^t S x = \| \Tr \xx\|_U^2$, these estimates imply that the
  Rayleigh quotient $x^t S x / x^t x$ is at most $\lambda_1 C_2^2$ and
  at least $C_1^2 \lambda_0/\cpi^2$. 
  Hence
  \begin{equation}
    \label{eq:kappa}
    \kappa(S) \le \frac{\lambda_1}{\lambda_0} \frac{C_2^2 \cpi^2}{C_1^2},
  \end{equation}
  where $\kappa(S)$ is the spectral condition number of $S$. This
  gives condition numbers comparable to other methods, as we shall see
  later in our examples.
\end{remark}

\section{First example: Laplace equation}
\label{section_Laplace}

The ideal DPG method for the Laplace equation was developed and
analyzed in~\cite{DemkoGopal:DPGanl}. In this section, we will set the
abstract forms and spaces of the previous section to those
from~\cite{DemkoGopal:DPGanl} and verify the hypotheses required to
apply Theorem~\ref{thm:approach}.  Roughly speaking, our main result
shows that if polynomials of degree $p$ are used to approximate the
solution of the Laplace equation, then a sufficient condition for
optimal convergence is that $T$ is approximated by polynomials of
degree $p+N$, where $N\ge 2$ is the space dimension. In the wording
of~\cite{DemkoGopal:DPGanl}, this means the ``enrichment degree''
should be chosen to be $N$. Ample numerical evidence, in support of
the choice of 2 as enrichment degree, was presented
in~\cite[\S~6.1]{DemkoGopal:DPGanl}, but all numerical experiments
were in the two-dimensional case.

In the remainder of this paper, we let $\om$ be a Lipschitz polyhedron
in $ \mathbb{R}^{N}$.  We denote by $\{\om_{h}\}_{h\in I}$ a family of
conforming shape regular simplicial finite element triangulations of
$\om$. The index $h$ now stands for the maximal diameter of simplexes
in $\om_h$.

\subsection{Infinite dimensional spaces}      

Let $\VVV = \RRR^N$. We use $L^2(\om,\VVV)$ to denote the set of
vector-valued functions whose components are square integrable.  We
set the trial and test spaces by
\begin{align*}
& U = L^{2}(\om;\mathbb{V})\times L^{2}(\om)\times H^{1/2}_{0}(\partial\om_h)\times H^{-1/2}(\partial\om_h),\\
& V = H(\text{div},\om_h)\times H^{1}(\om_h),
\end{align*}
where the ``broken'' Sobolev spaces (admitting interelement
discontinuities) are defined by $\Hdiv\oh = \{ \tau: \; \tau|_K \in
\Hdiv {K}, \;\forall K \in \oh\}$ and $H^1(\oh)  = \{ v: \; v|_K \in
H^1(K), \;\forall K \in \oh\}.$ They have the natural norms
\begin{align*}
\| v \|_{H^1(\oh)}^2
& = (v,v)_\oh + (\grad v,\grad v)_\oh,
\\
\| q \|_{\Hdiv\oh}^2
& = ( q,q)_\oh + ( \dive q, \dive q)_{\oh}.
\end{align*} 
The derivatives above, and in such notations throughout, are
calculated element by element and
\begin{align*}
  (r,s)_\oh =  \sum_{K\in\oh} (r,s)_K,
  \qquad
  \ip{ w, \ell}_{\d\oh} = \sum_{K\in\oh} \ip{ w,\ell}_{1/2,\d K}.
\end{align*}
and $\ip{ \cdot, \ell}_{1/2,\d K}$ denotes the action of a functional
$\ell$ in $H^{-1/2} (\d K)$.  We will also use $\| r \|_\oh$ to denote
the norm $(r,r)_\oh^{1/2}$.  The spaces of traces and fluxes are
defined by $ H_0^{1/2}(\d\om_h) = \{\eta\in \prod_K H^{1/2}(\d K):  \exists\, w \in H_0^1(\om)$
such that $\eta|_{\d K} = w|_{\d K} \; \forall K \in \oh \},$ and
$H^{-1/2}(\d \om_h) = \{ \eta \in \prod_K H^{-1/2}(\d K): \exists \,q
\in \Hdiv\om \text{ such that } \eta|_{\d K} = q\cdot n|_{\d K} \;
\forall K \in \oh \},$ with respective norms
\begin{align}
  \label{eq:quotient-h1}
  \| \hat u \|_{H_0^{1/2}(\d\oh)} 
  & = \inf
      \big\{ \| w \|_{H^1(\om)}:  \; \forall w \in H_0^1(\om)
      \text{ such that } \hat u|_{\d K} = w|_{\d K} \big\},
\\
\label{eq:quotient-hdiv}
  \|  \sgn \|_{H^{-1/2}(\d\oh)} 
  & =  \inf
      \big\{ \| q \|_{\Hdiv\om}:  \; \forall  q \in \Hdiv\om
      \text{ such that } \sgn |_{\d K} = q\cdot n|_{\d K} \big\}.
\end{align}
The spaces $U$ and $V$ are endowed with product norms, i.e.,
\begin{align*}
&\Vert (\sigma,u,\hat{u},\hat{\sigma}_{n})\Vert_{U}^2
= \Vert\sigma\Vert_{\om}^2+\Vert u\Vert_{\om}^2
+\Vert\hat{u}\Vert_{H^{1/2}(\partial\om_h)}^2
+\Vert \hat{\sigma}_n\Vert_{H^{-1/2}(\partial\om_h)}^2,
\\
&\Vert (\tau,v)\Vert_{V}^2=
\Vert\tau\Vert_{H(\text{div},\om_h)}^2
+\Vert v\Vert_{H^{1}(\om_h)}^2.
\end{align*}

\subsection{Forms}      

The ultraweak formulation of the Laplace equation derived
in~\cite{DemkoGopal:DPGanl} reads as follows: Find $\uu \equiv
(\sigma,u,\hat{u},\hat{\sigma}_{n}) \in U$ satisfying~\eqref{eq:1} for
$\vv \equiv (\tau,v) \in V$ where the forms $b(\cdot,\cdot)$ and
$l(\cdot)$ are set by
\begin{align*}
 b(\uu,\vv) & = (\sigma,\tau)_{\om}-(u,\dive \tau)_{\om_h}
+\left\langle \hat{u},\tau\cdot n\right\rangle_{\partial\om_h}
-(\sigma,\grad  v)_{\om_h}+\left\langle v, \hat{\sigma}_{n}\right\rangle_{\partial\om_h},
\\
l(\vv) & = (f,v)_{\om},
\end{align*}
for some $f$ in $L^2(\om)$. The $u$-component of $\uu$ solves the
Laplace equation with zero Dirichlet boundary conditions on $\d \om$.
For details, consult~\cite{DemkoGopal:DPGanl}.

\subsection{Discrete spaces}        \label{ssec:discrete}

Let us first establish notation for a few polynomial spaces that we
will use here and throughout.  Let $P_{p}(K)$ denote the space of
polynomials of degree at most $p$ on a simplex $K$. We write
$P_{p}(K;\mathbb{V})$ for vector valued functions whose components are
in $P_p(K)$.  Let $\triangle_{m}(K)$ denote the set of all
$m$-dimensional sub-simplices of $K$. Define
\begin{align*}
\mathring{P}_{p}(K) & = \{p_{p}\in P_{p}(K):\; p_{p}|_{\partial K}=0\},\\
P_{p}(\partial K) & = \{\mu : \mu |_{F}\in P_{p}(F),
                \;\forall F \in \triangle_{N-1}(K)\},\\
\tilde{P}_{p}(\partial K)
 & = P_{p}(\partial K) \cap \mathcal{C}^{0}(\partial K),
\end{align*}
where $\mathcal{C}^0(D)$ denotes the set of continuous functions on any
domain $D$.

Using these notations, we set the trial approximation space for the
DPG method by
\begin{alignat*}{3}
U_{h}  
& = \{(\sigma,u,\hat{u},\hat{\sigma}_{n})\in U 
& : \quad &  \sigma |_{K}\in P_{p}(K;\mathbb{V}), \;
           u|_{K}\in P_{p}(K),
\\
& & & \; \hat{u}|_{\partial K}\in \tilde{P}_{p+1}(\partial K), \; 
       \hat{\sigma}_{n}|_{\partial K}\in P_{p}(\partial K), \;\; \forall K \in \oh\}.
\end{alignat*}
The discrete test space is defined by $\Vhr = \Tr(U_h)$, so to
complete the prescription of the practical DPG method, we only need to
specify $\Vr$. Set
\begin{align}
\label{eq:VrLaplace}
\Vr & = \{ (\tau, v)\in V :\;
\tau |_{K}\in P_{r}(K;\mathbb{V}), \;
v|_{K}\in P_{r}(K), \; \forall K \in \oh\}.
\end{align}
where the degree $r \ge p+N$. Clearly, the 
application of $\Tr$, as defined by~\eqref{eq:Tr}, 
can proceed locally, element by element, since $\Vr$ has no
interelement continuity constraints.

\subsection{Verification of the assumptions}

To apply Theorem~\ref{thm:approach} to the above setting, we need to
verify its assumptions.
\begin{itemize}
\item Assumption~\eqref{eq:inj2} 
  is verified   by~\cite[Lemma~4.1]{DemkoGopal:DPGanl}.
\item Assumption~\eqref{eq:equiv2}
  is verified   by~\cite[Theorem~4.2]{DemkoGopal:DPGanl}.
\item Assumption~\eqref{eq:b-cnt} is easy to verify. 
  For example, to show the
  continuity of the term $\ip{ \hat u, \tau \cdot n}_{\d\oh}$, we let $w \in
  H^1(\om)$ be any extension of $\hat u$ and observe that 
  \begin{align*}
    \ip{\hat u, \tau \cdot n }_{\d \oh}
    & = (\grad w, \tau)_\oh + (w,\dive \tau)_\oh
     \le \| w \|_{H^1(\om)} \| \tau \|_{\Hdiv\oh}.
  \end{align*}
  Taking the infimum over all such extensions~$w$, we obtain
  \[
  \ip{\hat u, \tau \cdot n }_{\d \oh} 
  \le \| \hat u \|_{H_0^{1/2}(\d \oh)}   \| \tau \|_{\Hdiv\oh}.
  \]
  The other terms in the bilinear form are similar or simpler.

\item Assumption~\eqref{eq:pi} is verified below.
\end{itemize}

An operator $\vpi$ satisfying~\eqref{eq:pi} will be constructed in the
form $\vpi \vv = (\vpid \tau, \vpig v)$.  We construct the operator
$\vpig$ in Lemma~\ref{lem:piG} below, and we construct the operator
$\vpid$ in Lemma~\ref{lem:projection_div}.  But first, we need the
following intermediate result.  Let $\Brg(K)=\{p_{r}\in
P_{r}(K):p_{r}|_{E}=0, \forall E\in\triangle_{N-2}(K)\}$ and $h_K =
\mathrm{diam}(K)$.  Hereon we use $c$ and $C$ to denote generic
constants (whose value at different occurrences may differ) independent
of $h_K$, but possibly dependent on the shape regularity of $K$ and
the polynomial degree~$p$. We also let $\ip{\cdot,\cdot}_{\d K}$
denote the $L^2(\d K)$-inner product.

\begin{lemma}
\label{lemma_h1}
Let $r=p+N$. Then, for every $v\in H^{1}(K)$, there is a unique
$\vpi_r^0v\in \Brg(K)$ satisfying
\begin{subequations}
\begin{align}
\label{projection_H1_1}
(\vpi_r^0 v - v,q_{p-1})_{K}
& =0,\quad\forall q_{p-1}\in P_{p-1}(K),\\
\label{projection_H1_2}
\langle \vpi_r^0 v - v, \mu_{p}\rangle_{\partial K}
& =0,\quad\forall \mu_{p}\in P_{p}(\partial K),\\
\label{projection_H1_3}
\Vert \vpi_r^0v\Vert_{K}+h_{K}\Vert \grad  \vpi_r^0v \Vert_{K}
& \leq 
C\left(\Vert v\Vert_{K}+h_{K}\Vert \grad  v\Vert_{K} \right).
\end{align}
\end{subequations}
\end{lemma}
\begin{proof}
  First, to see that the number of the equations in
  (\ref{projection_H1_1})-(\ref{projection_H1_2}) equal $\dim 
  \Brg(K)$, observe that 
  \begin{equation}
    \label{eq:3}
    \dim \Brg(K) = 
    \dim\mathring{P}_{r}(K)
    +
    \sum_{F \in \triangle_{N-1}(K)}\dim\mathring{P}_r(F).
  \end{equation}
  Let $b_K$ and $b_F$ denote the product of all barycentric
  coordinates that do not vanish everywhere on $K$ and $F$, resp. Then
  $\mathring{P}_{r}(K) = b_K P_{r-N-1}(K)$ and $\mathring{P}_r(F) =
  b_F P_{r-N}(F)$. Therefore, by our choice of $r$, we have
  $\dim\mathring{P}_{r}(K) = \dim P_{p-1}(K)$ and
  $\dim\mathring{P}_r(F) = \dim P_p(F)$. It then follows
  from~\eqref{eq:3} that \eqref{projection_H1_1}-\eqref{projection_H1_2}
    is a square system for $\vpi_r^0 v$.

    Hence, to prove that
    \eqref{projection_H1_1}-\eqref{projection_H1_2} has a unique
    solution, it suffices to prove that if $v=0$, then $\vpi_r^0v=0$.
    Since $\vpi_r^0v\in \Brg(K)$, on any face $F \in \tr_{N-1}(K)$, we
    may write $(\vpi_r^0v)|_{F} = b_F w_p$ for some $w_p \in P_p(F)$.
    But then,~\eqref{projection_H1_2} implies that $\vpi_r^0v$ must
    vanish on $\d K$, so $\vpi_r^0 v = b_K z_{p-1}$ for some $z_{p-1}
    \in P_{p-1}(K)$.  Then (\ref{projection_H1_1}) implies that
    $\vpi_r^0v=0$ on $K$.

    Finally, one can prove (\ref{projection_H1_3})  using a standard affine
    mapping argument. 
\end{proof}

\begin{lemma}
\label{lem:piG}
Let $r=p+N$. Define $\vpig v=\vpi^0_{r}(v-\overline{v})+\overline{v}$,
where $\overline{v}|_{K}=\vert K\vert^{-1}\int_{K}v$. 
Then
\begin{subequations}
\label{eq:projectionGrad}
\begin{align}
\label{projection_H1_1_modified}
(\vpig v - v,q_{p-1})_{K}
& =0,
&& \quad\forall q_{p-1}\in P_{p-1}(K),\\
\label{projection_H1_2_modified}
\langle \vpig v - v, \mu_{p}\rangle_{\partial K}
& =0,
&& \quad\forall \mu_{p}\in P_{p}(\partial K),\\
\label{projection_H1_3_modified}
\Vert \vpig v \Vert_{H^{1}(K)}
 \leq 
C&\Vert v\Vert_{H^{1}(K)},
&&\quad\forall v\in H^{1}(K).
\end{align}
\end{subequations}
\end{lemma}
\begin{proof}
  Obviously, $\vpig v-v=(\vpi^0_r - I)(v-\overline{v})$.  Hence,
  (\ref{projection_H1_1_modified}) and
  (\ref{projection_H1_2_modified}) follows
  from~\eqref{projection_H1_1} and~\eqref{projection_H1_2} of
  Lemma~\ref{lemma_h1}.  It remains to prove (\ref{projection_H1_3_modified}). By~\eqref{projection_H1_3} and the \Poincare-Friedrichs inequality, 
\begin{align*}
\Vert \vpig v\Vert_{K} & \leq\Vert\overline{v}\Vert_{K}+\Vert \vpi^0_r(v-\overline{v})\Vert_{K}
\\
 & \leq \Vert \overline{v}\Vert_{K}+C\left(\Vert v-\overline{v}\Vert_{K} +h_{K}\Vert\grad  (v-\overline{v})\Vert_{K}\right)
\\
 & \leq C\left(\Vert v\Vert_{K} + h_{K}\Vert \grad  v\Vert_{K}\right),
\qquad\text{ and}
\\
h_{K}\Vert\grad  \vpig v \Vert_{K} 
& = h_{K}\Vert \grad  \vpi^0_r(v-\overline{v})\Vert_{K}\\
 & \leq C\left(\Vert v-\overline{v}\Vert_{K} + h_{K}\Vert \grad  (v-\overline{v})\Vert_{K}\right) 
\\
 & \leq Ch_{K}\Vert \grad  v\Vert_{K}.
\end{align*}
Canceling out $h_K$ and adding, (\ref{projection_H1_3_modified}) follows.
\end{proof}

\begin{lemma}
\label{lem:projection_div}
There is an operator $\vpid : \Hdiv K \mapsto P_{p+2}(K;\mathbb{V})$
such that for every $\tau\in \Hdiv K$, we have 
\begin{subequations}
\label{projection_div}
\begin{align}
\label{projection_div_1}
(\vpid\tau,q_{p})_{K}
& = (\tau,q_{p})_{K}, && \forall q_{p}\in P_{p}(K;\mathbb{V}),\\
\label{projection_div_2}
\langle \vpid\tau\cdot n,\mu_{p+1}\rangle_{\partial K}  
&=
\langle \mu_{p+1}, \tau\cdot n\rangle_{1/2, \partial K}
&& \forall \mu_{p+1}\in\tilde{P}_{p+1}(\partial K),\\
\label{projection_div_3}
\Vert\vpid\tau\Vert_{\Hdiv{K}} 
& \leq C\Vert\tau\Vert_{\Hdiv{K}}. 
\end{align}
\end{subequations} 
\end{lemma}

\begin{proof}

  We will first construct the operator on the unit simplex $\hat K$ in
  $\RRR^N$.  Recalling  the
  notations in \S\ref{ssec:discrete}, define $P_{p}^{\bot}(\partial
  \hat{K})$ to be the $L^{2}(\partial \hat{K})$-orthogonal complement
  of $\tilde{P}_{p}(\partial \hat{K})$ in $P_{p}(\partial \hat{K})$, and 
  \begin{equation*}
    B_{p+2}^{\text{div}}(\hat{K})=\{\hat{\tau}\in P_{p+1}(\hat{K};\mathbb{V})+\hat{x}P_{p+1}(\hat{K}):\langle \hat{p}_{\bot},
\hat{\tau}\cdot \hat{n}\rangle_{\partial \hat{K}}=0,
\quad\forall \hat{p}_{\bot}\in P_{p+1}^{\bot}(\partial \hat{K})\}.
\end{equation*} 
We  construct an operator $\hat{\varPi}_{p+2}^{\text{div}}$ 
mapping $H(\text{div},\hat{K})$ into $B_{p+2}^{\text{div}}(\hat{K})$ by
\begin{subequations}
\begin{align}
\label{projection_div_reference_1}
(\hat{\varPi}_{p+2}^{\text{div}}\hat{\tau},\hat{q}_{p})_{\hat{K}} 
& = (\hat{\tau},\hat{q}_{p})_{\hat{K}},
&&\forall \hat{q}_{p}
\in P_{p}(\hat{K};\mathbb{V}),\\
\label{projection_div_reference_2}
\langle \hat{\varPi}_{p+2}^{\text{div}}\hat{\tau}\cdot \hat{n},\hat{\mu}_{p+1}\rangle_{\partial \hat{K}}
& =
\ip{\hat{\mu}_{p+1}, \hat{\tau}\cdot \hat{n}}_{1/2,\d K}
&&\forall \hat{\mu}_{p+1}\in\tilde{P}_{p+1}(\partial \hat{K}).
\end{align}
\end{subequations}
We claim that
(\ref{projection_div_reference_1})--(\ref{projection_div_reference_2})
uniquely determine $\hat{\varPi}_{p+2}^{\text{div}}\hat{\tau} \in
B_{p+2}^{\text{div}}(\hat{K})$.  Indeed, if their right hand sides
vanish, then since $\hat{\varPi}_{p+2}^{\text{div}}\hat{\tau}$ is in
$B_{p+2}^{\text{div}}(\hat{K})$, we find that
$\hat{\varPi}_{p+2}^{\text{div}}$ is a function in the Raviart-Thomas
space whose canonical degrees of freedom vanish (see
e.g.,~\cite[Definition~5]{Nedel80}), so
$\hat{\varPi}_{p+2}^{\text{div}}\hat{\tau}=0$.  Hence
(\ref{projection_div_reference_1})--(\ref{projection_div_reference_2})
uniquely defines $\hat{\varPi}_{p+2}^{\text{div}}\hat{\tau}$.

Now, we define $\vpid$ on any general simplex $K$ by mapping
$\hat{\varPi}_{p+2}^{\text{div}}\hat{\tau}$ from $\hat K$ using the
Piola transform, as follows.  Let $G_{K}$ be the affine homeomorphism
from $\hat{K}$ onto $K$ and let $A$ denote its \Frechet\ derivative.
Given any $\tau \in \Hdiv K$, let $\hat \tau (\hat x) $ in $\Hdiv
{\hat K}$ be defined by $\tau \circ G_K = (\det A)^{-1}A \hat \tau.$
Then, define $\vpid\tau$ by
\begin{equation*}
\vpid\tau(x) = \dfrac{A}{\det A}\hat{\varPi}_{p+2}^{\text{div}}
\hat{\tau}(\hat{x}),\quad\text{ with }x=G_{K}(\hat{x}).
\end{equation*} 
%
We will now show that this $\vpid \tau$ satisfies the three
properties in~\eqref{projection_div}.

First, observe that~\eqref{projection_div_reference_1}
and~\eqref{projection_div_reference_2} imply the corresponding
identities on $K$, namely,
\begin{align*}
  (\vpid\tau - \tau, \, A^{-t}\hat{q}_{p} \circ G_{K}^{-1} )_{K}  
  & = 0,
  &&\forall \hat{q}_{p}   \in P_{p}(\hat{K};\mathbb{V}),
  \\
    \langle \vpid\tau\cdot n ,
  \, 
  \hat{\mu}_{p+1} \circ G_{K}^{-1} \rangle_{\partial K}
  & 
  = \langle \hat{\mu}_{p+1} \circ G_{K}^{-1},
  \tau\cdot n \rangle_{1/2, \partial K}
  && 
  \forall \hat{\mu}_{p+1}\in\tilde{P}_{p+1}(\partial \hat{K}).
\end{align*}
This implies~\eqref{projection_div_1} and~\eqref{projection_div_2}.

It only remains to prove~\eqref{projection_div_3}.  We do this in two
steps. First, we prove an $L^2(K)$ bound using
the Piola map's well-known estimates for shape regular meshes,
namely
\begin{gather*}
c
\Vert\hat{\tau}\Vert_{\hat{K}}
\leq  
\Vert\tau\Vert_{K}\dfrac{\vert K\vert^{1/2}}{h_{K}}
\leq C \Vert\hat\tau \Vert_{\hat{K}},
\\
c
\Vert\mathop{\dive}\hat{\tau}\Vert_{\hat{K}}
 \leq 
 \Vert\dive\tau\Vert_{K}\vert K\vert^{1/2}
\leq C
\Vert \dive\hat\tau\Vert_{\hat{K}}.
\end{gather*}
Together with the fact that  $\hat{\varPi}_{p+2}^{\text{div}}$ is a continuous operator on $\Hdiv{\hat K}$, we obtain
\begin{equation}
\label{eq:5}
\Vert\vpid\tau\Vert_{K}+h_{K}\Vert \dive\vpid\tau\Vert_{K}\leq
C\left(\Vert\tau\Vert_{K}+h_{K}\Vert\dive \tau\Vert_{K}\right).
\end{equation}
In particular, this proves the $L^2(K)$-bound $\Vert \vpid\tau\Vert_{K}\leq
C\Vert\tau\Vert_{\Hdiv{K}}$.

Next, we prove a better bound on the divergence norm $\Vert
\dive\vpid\tau\Vert_{K}$
by showing that
\[
\dive (\vpid\tau)=\varPi_{p+1}\dive\tau
\] 
where $\varPi_{p+1}$ is the $L^{2}(K)$-orthogonal projection onto
$P_{p+1}(K)$.  Indeed, 
for any $\omega_{p+1}\in P_{p+1}(K)$, we have, due
to~\eqref{projection_div_1} and~\eqref{projection_div_2}, that
\begin{align*}
(\dive(\vpid\tau),\omega_{p+1})_{K} & =-(\vpid\tau,\grad \omega_{p+1})_{K}
+\langle (\vpid\tau)\cdot n,\omega_{p+1}\rangle_{\partial K}\\
 & = -(\tau,\grad \omega_{p+1})_{K}+
\ip{\omega_{p+1}, \tau\cdot n}_{1/2,\d K}
 \\
 & = (\dive\tau,\omega_{p+1})_{K}.
\end{align*}
Hence, 
\begin{equation}
\label{eq:4}
\Vert\dive(\vpid\tau)\Vert_{K}=\Vert \varPi_{p+1}\dive\tau\Vert_{K}\leq \Vert\dive\tau\Vert_{K}.
\end{equation}
Estimates~\eqref{eq:5} and~\eqref{eq:4} prove~\eqref{projection_div_3}.
\end{proof}

Now we are ready to apply Theorem~\ref{thm:approach} to obtain a convergence result for the practical DPG method for the Laplace equation.

\begin{theorem}
  \label{thm:laplace}
  Let $r\ge p+N$. Then the exact and discrete solutions for the DPG
  method for the Laplace's equation, namely $ \uu =
  (\sigma,u,\hat{u},\hat{\sigma}_{n})$ and $ \uuh =
  (\sigma_h,u_h,\hat{u}_h,\hat{\sigma}_{n,h})$, satisfy
  \begin{multline*}
    \| \sigma - \sigma _h \|_{L^2(\om)} + 
    \| u - u_h \|_{L^2(\om)} + 
    \| \hat{u} - \hat{u}_h \|_{H_0^{1/2}(\d\oh)} + 
    \| \hat\sigma_n - \hat{\sigma}_{n,h} \|_{H^{-1/2}(\d\oh)}
    \\
    \le C
    \inf_{ (\rho_h,w_h,\hat{w}_h,\hat{\eta}_{h}) \in U_h }
    \bigg(
    \| \sigma - \rho_h \|_{L^2(\om)} + 
      \| u - w_h \|_{L^2(\om)} 
      \\ + 
      \| \hat{u} - \hat{w}_h \|_{H_0^{1/2}(\d\oh)} + 
      \| \hat\sigma_n - \hat{\eta}_{n,h} \|_{H^{-1/2}(\d\oh)}
    \bigg).
  \end{multline*}
\end{theorem}
\begin{proof}
  As already observed, we have verified the first three assumptions of
  Theorem~\ref{thm:approach}. To verify Assumption~\eqref{eq:pi}, let
  $\vv = (\tau, v)$ and set $\vpi \vv = (\vpid \tau, \vpig v)$. The
  continuity estimates of $\vpid$ and $\vpig$ of Lemmas~\ref{lem:piG}
  and~\ref{lem:projection_div} (namely \eqref{projection_H1_3_modified}
  and~\eqref{projection_div_3}) show that~\eqref{eq:pi-est} holds.  To
  see that~\eqref{eq:pi-iden} also holds, observe that the identities
  of these lemmas also imply
  \begin{align*}
    (\rho_h, \tau - \vpid \tau )_\om & =0, 
    &
    (w_h, \dive (\tau - \vpid\tau))_\oh & =0,
    \\
    \ip{ \hat{w}_h, (\tau - \vpid \tau)\cdot n}_{\d\oh}&=0,
    &
    (\rho_h, \grad (v - \vpig v) )_\oh & = 0,
    \\
    \ip{ v - \vpig v , \hat{\eta}_h}_{\d\oh} & = 0,
  \end{align*}
  for all $(\rho_h,w_h,\hat{w}_h,\hat{\eta}_{h}) \in U_h.$ While the
  identities above on the left follow from the identities
  of~\eqref{eq:projectionGrad} and~\eqref{projection_div}, those on
  the right are proved by integration by parts. Together these
  identities imply that $b( \ww, \vv - \vpi \vv)=0$ for all $\ww \in
  U_h$, so Assumption~\eqref{eq:pi} is satisfied.
\end{proof}

\begin{remark}[Enrichment degree]
  The above arguments point to the potential of choosing different
  enrichment degrees for the scalar and flux components of the test
  space. We have in fact proved that if, in place of the $\Vr$ set
  in~\eqref{eq:VrLaplace}, we revise our choice of $\Vr$ to
  \begin{align*}
    \Vr & = \{ (\tau, v)\in V :\;
    \tau |_{K}\in P_{p+2}(K;\mathbb{V}), \;
    v|_{K}\in P_{p+N}(K), \; \forall K \in \oh\},
  \end{align*}
  then, we obtain the {\em same} convergence result.  Obviously, the
  revised $\Vr$ defines a smaller space if $N\ge 3$. The present DPG
  software packages are set to approximate all components of $T$ by
  polynomials of the same degree~$r$. Our results indicate that this
  is unnecessary.
\end{remark}

As an example of how Theorem~\ref{thm:laplace} implies $h$-convergence
rates, we state the following.

\begin{corollary}[Convergence rates]
  \label{cor:rates}
  Let $h = \max_{K\in\oh} \diam(K)$,  $N=2$ or $3$, 
  and let the assumptions of
  Theorem~\ref{thm:laplace} hold. Then
  \begin{multline*}
    \| \sigma - \sigma_h \|_{L^2(\om)} 
    +
    \| u - u_h \|_{L^2(\om)} 
    + 
    \| \hat u - \hat u_h \|_{H_0^{1/2}(\d\oh)}
    +
    \| \sgn - \sgnh \|_{H^{-1/2}(\d\oh)} 
    \\
    \le \,
    C  h^{s} \big(  \|u \|_{H^{s+1}(\om)} + \| \sigma \|_{H^{s+1}(\om)} \big),
 \end{multline*}
 for all $1/2 <s\le p+1$.
\end{corollary}
\begin{proof}
  The proof proceeds by bounding the infimum over
  $(\rho_h,w_h,\hat{w}_h,\hat{\eta}_{h})\in U_h$ in
  Theorem~\ref{thm:laplace}.  It is standard to bound the first two
  terms in the infimum, so we will only explain how to bound the next
  two terms.  It is well-known (see,
  e.g.,~\cite[Theorem~8.1]{DemkoGopalSchob11}) that there are
  interpolants $\gpi u \in H_0^1(\om)$ and $\dpi \sigma \in \Hdiv\om$,
  such that $\gpi u|_K \in P_{p+1}(K)$, $\dpi\sigma|_K \in \vx P_p(K)
  + P_p(K)^3$ for all $K\in\oh$, and the interpolation errors satisfy
  \begin{subequations}
    \label{eq:8}
    \begin{align}
      \| u - \gpi u \|_{H^1(\om)}
      & \le 
      C h^s
      | u |_{H^{s+1}(\om)},
      \qquad (1/2< s\le p+1),
      \\
      \| \sigma - \dpi \sigma \|_{\Hdiv\om}
      & \le 
      Ch^s 
      |\sigma |_{H^{s+1}(\om)},
      \qquad (0<s\le p+1).
    \end{align}
  \end{subequations}
  Let $\hat t_h$ denote the trace of $\gpi u$ on
  $\d\oh$. Then, 
  \begin{align*}
    \inf_{\hat w_h} \| \hat{u} - \hat{w}_h \|_{H_0^{1/2}(\d\oh)} 
    & \le
    \| \hat u - \hat t_h \|_{H_0^{1/2}(\d\oh)}\le  \| u - \gpi u \|_{H^1(\om)}. 
  \end{align*}
  The last inequality is obtained by observing that $\hat u$ is the
  trace of $u$ on $\d\oh$ and bounding the infimum in
  definition~\eqref{eq:quotient-h1}. In a similar fashion, we can
  estimate the last term in Theorem~\ref{thm:laplace} by $\| \sigma -
  \dpi \sigma \|_{\Hdiv\om}.$ The interpolation error
  estimates~\eqref{eq:8} then finish the proof.
\end{proof}

To conclude this section, we prove that the condition number of the
stiffness matrix of the DPG method is no worse than other standard
methods -- see Remark~\ref{rem:cond-abd} for the definition of the
stiffness matrix with respect to a basis $\{\bb_i\}$.  Consider, for
definiteness, the three-dimensional tetrahedral case. We tacitly
assume that the basis functions $\bb_i$ are local, and obtained, as in usual
finite element practice, by mapping from the (reference) unit
simplex.  For example, a basis for the trial space for the
numerical traces is built using a local basis $\{ e_j\}$ for $\tilde
P_{p+1}(\d K)$, which in turn is obtained by mapping over a basis
$\{\hat e_j\}$ for $\tilde P_{p+1}(\d \hat K)$ (where $e_j = \hat e_j
\circ G_K^{-1}$ and we use the other mapping notations in the proof of
Lemma~\ref{lem:projection_div}).  Consequently, if $\hat s = \sum_j
s_j \hat e_j$ is the basis expansion for any $\hat s \in \tilde
P_{p+1}(\d K)$, then by the equivalence of norms in finite dimensional
spaces
 \begin{equation}
   \label{eq:14}
    c \sum_j |s_j|^2 
    \le
    \splitinf{\hat e \in P_{p+1}(\hat K),}{ (\hat e - \hat s)|_{\d\hat K}=0}
    \| \hat e\|_{H^1(\hat K)}^2
    \le C \sum_j |s_j|^2. 
 \end{equation}
 Such arguments will be used in the following proof without
 further explanation.


 \begin{theorem}[Conditioning]
   \label{thm:cond}
   Suppose $\oh$ is a quasiuniform tetrahedral mesh and the
   assumptions of Theorem~\ref{thm:laplace} hold.  Then the spectral
   condition number of the stiffness matrix $S$ of the DPG method
   satisfies
   \[
   \kappa(S) \le C h^{-2}.
   \]
 \end{theorem}
 \begin{proof}
   Let us apply~\eqref{eq:kappa}. We have already shown above that
   $C_1, C_2$ and $\cpi$ are independent of $h$. Hence it only
   suffices to find the dependence of $\lambda_0$ and $\lambda_1$ on
   $h$ in~\eqref{eq:lambdas}. 

   Let $\xx = (\rho, w, \hat z, \hat \eta)$ in $U_h$.  As a first step
   to bound the norm of $\hat z$, we recall that the existence of an
   $H^1(\hat K)$ polynomial extension~\cite{DemkoGopalSchob08} implies
   that for any $S$ in $\tilde P_{p+1}(\d \hat K)$,
   \begin{align*}
     \splitinf{\hat e_p \in P_{p+1}(\hat K),}{ \hat e_p|_{\d \hat K} = S}
     \| \hat e_p \|_{\hat K}^2 + \| \grad {\hat e_p} \|_{\hat K}^2 
     \le 
     C \splitinf{\hat e \in H^1(\hat K),}{ \hat e|_{\d \hat{•} K} =S}
       \| \hat e \|_{\hat K}^2 + \| \grad {\hat e} \|_{\hat K}^2. 
  \end{align*}
  Mapping to $K$ using the affine homeomorphism $G_K$
  and scaling both sides by $|K| $, we obtain 
  \begin{align}
    \label{eq:6}
    \splitinf{e_p \in P_{p+1}( K),}{e_p|_{\d  K} =  s  }
     \|  e_p \|_K^2 + h_K^2 \| \grad e_p \|_K^2 
    \le 
    C 
    \splitinf{ e \in H^1( K),}{  e|_{\d  K} =   s }
    \| e \|_K^2 + h_K^2  \| \grad e \|_K^2.
  \end{align}
  where  $s = S \circ G_K^{-1}$.
  Let us denote the function which achieves the left infimum by
  $\Epgrad s$. Applying the above inequality, element by element,
  with $s$ replaced by $\hat z$, and using $h_K \le \diam\om$,
  we have proved that
  \[
  \| \Epgrad \hat z \|_\om^2 \le  C \max(1,\diam\om)^2
  \splitinf{ e \in H^1( \om),}{  (e-\hat z)|_{\d  K} =  0 }
  \left( \| e \|_\om^2 +   \| \grad e \|_\om^2\right)
  \le C  \| \hat z \|_{H_0^{1/2}(\d\oh)}^2.
  \]
  Thus, 
  \[
    c\| \Epgrad \hat z \|_\om^2
    \le \| \hat z \|_{H_0^{1/2}(\d\oh)}^2
    \le
    \| \Epgrad \hat z \|_{H^1(\om)}^2,
  \]
  where the upper inequality is obvious from the definition of the
  $H_0^{1/2}(\d\oh)$-norm.  By an inverse inequality, 
  \begin{equation}
    \label{eq:15}
    c\| \Epgrad \hat z \|_\om^2
    \le \| \hat z \|_{H_0^{1/2}(\d\oh)}^2
    \le
    C h^{-2}\| \Epgrad \hat z \|_{\om}^2.
  \end{equation}

  A similar argument using the $\Hdiv {\hat K}$ polynomial extension
  in~\cite{{DemkoGopalSchob11}}, gives
  \begin{equation}
    \label{eq:9}
    c\| \Epdiv\hat \eta \|_\om^2
    \le \| \hat \eta \|_{H^{-1/2}(\d\oh)}^2
    \le
    Ch^{-2}\| \Epdiv \hat \eta \|_{\om}^2.
  \end{equation}
  Combining~\eqref{eq:15} and~\eqref{eq:9}, we have 
  \begin{equation}
     \label{eq:12}
     c \| \xx \|_0^2 \le \| \xx \|_U^2 \le Ch^{-2} \| \xx \|_0^2,\qquad 
     \forall \xx \in U_h,
  \end{equation}
  where $ \| \xx \|_0^2 = \| \rho \|_\om^2 + \| w \|_\om^2 + \|\Epgrad
  \hat z \|_\om^2 + \| \Epdiv\hat\eta \|_\om^2.$

To prove~\eqref{eq:lambdas}, consider 
the coefficients in 
the basis expansion of $\xx$. If $z_j$'s denote the
coefficients in a basis expansion of the $\hat z|_{\d K}$, then
using~\eqref{eq:14} and the minimization property of $\Epgrad \hat z$,
we obtain
\[
c \sum_j |z_j|^2 \le \frac{1}{|K|}
\left( \| \Epgrad \hat z \|_K^2 + h_K^2 \| \grad \Epgrad \hat z \|_K^2
\right)
\le C \sum_j |z_j|^2.
\]
By an inverse inequality, 
\[
c \sum_j |z_j|^2 \le \frac{1}{|K|}
 \| \Epgrad \hat z \|_K^2
\le C \sum_j |z_j|^2.
\]
A similar estimate holds for $\Epdiv \hat \eta$. Combining these with
the obvious estimates for the coefficients in the expansion of $\rho$
and $w$, we find that
\begin{equation}
  \label{eq:11}
c \| x \|_{\ell^2}^2  \min_{K\in\oh}|K| 
\le \| \xx \|_0^2 
\le 
C \| x \|_{\ell^2}^2  \max_{K\in\oh} |K|.
\end{equation}
Clearly, inequalities~\eqref{eq:11} and~\eqref{eq:12} imply~\eqref{eq:lambdas}
with $ \lambda_0 =c \min_{K\in\oh} |K|$ and $\lambda_1 = C
h^{-2} \max_{K\in\oh} |K|,$ thus completing the proof.
\end{proof}

\section{Second example: Linear elasticity}
\label{sec:elas}

Two ideal DPG methods for the linear elasticity equation were
developed and analyzed in~\cite{BramDemkoGopalQiu:DPGelas}. The two
methods are equivalent for homogeneous isotropic materials.  Among
their many interesting properties is their robustness with respect to
the Poisson ratio, i.e., the method is locking-free. In this section,
we will consider the practical version of one of these two methods and
prove its optimal convergence. We proceed as in the previous example,
by first setting the abstract forms and spaces to those specific to
this method then proceed to verify the hypotheses required to apply
Theorem~\ref{thm:approach}.  In this section, we restrict to $N=2$ or
$3$.  The results and the analysis are similar to those in
Section~\ref{section_Laplace}, so we will be brief.

\subsection{The spaces}

We set the trial and test spaces by
\begin{align*}
& U = L^{2}(\om;\mathbb{M})\times L^{2}(\om;\mathbb{V})\times 
H^{1/2}_{0}(\partial\om_h;\mathbb{V})\times H^{-1/2}(\partial\om_h;\mathbb{V})\times \mathbb{R},\\
& V = H(\text{div},\om_h;\mathbb{S})\times H^{1}(\om_h;\mathbb{V})\times L^{2}(\om ;\mathbb{K}) \times \mathbb{R},
\end{align*}
where $\mathbb{M}=\mathbb{R}^{N\times N}$, $\mathbb{S}$ consists of
symmetric matrices in $\mathbb{M}$, and $\mathbb{K}$ consists of
skew-symmetric matrices in $\mathbb{M}$.  The trial and test spaces
are normed by
\begin{align*}
&\Vert (\sigma,u,\hat{u},\hat{\sigma}_{n},\alpha)\Vert_{U}^{2} = \Vert\sigma\Vert_{\om}^{2}+\Vert u\Vert_{\om}^{2}
+\Vert\hat{u}\Vert_{H^{1/2}(\partial\om_h)}^{2}+\Vert \hat{\sigma}_n\Vert_{H^{-1/2}(\partial\om_h)}^{2}+\vert\alpha\vert^{2},\\
&\Vert (\tau,v,q,\beta)\Vert_{V}^{2}=\Vert\tau\Vert_{H(\text{div},\om_h)}^{2}+\Vert v\Vert_{H^{1}(\om_h)}^{2}+\Vert  q\Vert_{\om}^{2}
+\vert\beta\vert^{2}.\\
\end{align*}

\subsection{The forms}      

The (second) ultraweak formulation derived
in~\cite{BramDemkoGopalQiu:DPGelas} reads as follows: Find $\uu \equiv
(\sigma,u,\hat{u},\hat{\sigma}_{n},\alpha) \in U$
satisfying~\eqref{eq:1} for all $\vv \equiv (\tau,v,q,\beta) \in V$ where
the forms $b(\cdot,\cdot)$ and $l(\cdot)$ are set by
\begin{align*}
  b(\uu,\vv) 
  &
   = 
  (\mathcal{A}\sigma,\tau)_{\om_{h}}+(u,\dive\tau)_{\om_{h}}
  -\langle\hat{u},\tau\,n\rangle_{\partial\om_{h}}+Q_{0}^{-1}(\alpha I,\mathcal{A}\tau)_{\om} \\
  \nonumber
  & \qquad\qquad +(\sigma,\grad  v)_{\om_h}+(\sigma,q)_{\om}
  -\langle v,\hat{\sigma}_{n}\rangle_{\partial\om_{h}} \\
  \nonumber
  & \qquad\qquad +Q_{0}^{-1}(\mathcal{A}\sigma,\beta I)_{\om}
  \\
  \nonumber
  l(\vv) & = (f,v)_{\om},
\end{align*}
for some $f$ in $L^2(\om;\mathbb{V})$.  Here, $\mathcal{A}$ is the
generalized compliance tensor (see e.g.,
\cite[Remark~2.1]{BramDemkoGopalQiu:DPGelas}) and $Q_0$ is the
essential infimum of the trace of the matrix $\mathcal A(x) I$ over
$x\in \om$.  Throughout, we assume that $\mathcal{A}$ is element-wise
constant.  We note that above and throughout, the inner products of
matrix-valued functions, such as $(\sigma, \tau)_K$, are computed by
integrating the Frobenius product of the two matrices.
 
It is easy to see that the resulting $\sigma$ and $u$
satisfies $ \mathcal{A}\sigma = \veps(u)$, where $ \veps(u)= (\grad u + (\grad
u)')/2,$ and $\dive \sigma = f$ on $\om$, and $u=0$ on $\d \om$, and $\alpha = 0$.  For
details, consult~\cite{BramDemkoGopalQiu:DPGelas}.

\subsection{Discrete spaces}      

Symmetric, skew-symmetric, and general matrix-valued functions whose
entries are in $P_p(K)$ are denoted by
$P_{p}(K;\mathbb{S}),P_{p}(K;\mathbb{K})$, and $P_{p}(K;\mathbb{M})$,
resp.  Using these notations, we set the trial approximation space for
the DPG method by
\begin{alignat*}{3}
U_{h}  
& = \{(\sigma,u,\hat{u},\hat{\sigma}_{n},\alpha)\in U 
& : \quad &  \sigma |_{K}\in P_{p}(K;\mathbb{M}), \;
           u|_{K}\in P_{p}(K;\mathbb{V}),
\\
& & & \; \hat{u}|_{\partial K}\in \tilde{P}_{p+1}(\partial K;\mathbb{V}), \; 
       \hat{\sigma}_{n}|_{\partial K}\in P_{p}(\partial K;\mathbb{V}), \;\alpha\in\mathbb{R},  \;\; \forall K \in \oh\}.
\end{alignat*}
The discrete test space is defined by $\Vhr = \Tr(U_h)$, so to
complete the prescription of the practical DPG method, we only need to
specify $\Vr$. Set
\begin{alignat*}{3}
\Vr & = \{ (\tau, v, q, \beta)\in V :
& \quad &
\tau |_{K}\in P_{r}(K;\mathbb{S}), \;
v|_{K}\in P_{r}(K;\mathbb{V}), 
\\
& & & 
q|_{K}\in P_{p}(K;\mathbb{K}), \; \beta\in\mathbb{R}, \; 
\forall K \in \oh\},
\end{alignat*}
for some integer $r\geq p+N$.

\subsection{Verification of the assumptions}

To apply Theorem~\ref{thm:approach} to the above setting, we need to
verify its assumptions.
\begin{itemize}
\item Assumption~\eqref{eq:inj2} 
  is verified   by~\cite[Lemma~5.2]{BramDemkoGopalQiu:DPGelas}.
\item Assumption~\eqref{eq:equiv2}
  is verified   by~\cite[Lemma~5.3]{BramDemkoGopalQiu:DPGelas}.
\item Assumption~\eqref{eq:b-cnt} can be easily verified, as in the
  case of the Laplace equation.
\item Assumption~\eqref{eq:pi} is verified next.

\end{itemize}

Let $\vv = (\tau, v, q, \beta)\in V$. The
operator $\vpi$ satisfying~\eqref{eq:pi} will take the form
\begin{equation}
  \label{eq:Pi-S}
  \vpi \vv = (\vpidS \tau, \vpig v, \varPi_{p}^\KKK q, \beta).
\end{equation}  
We set $\vpig v$ to be the one defined in Lemma~\ref{lem:piG}, but
applied component by component, to the vector valued function~$v$.
The operator $\varPi_{p}^\KKK$ is simply the $L^{2}$-orthogonal projection
onto $\{q\in L^{2}(\om;\mathbb{K}): q|_{K}\in P_{p}(K;\mathbb{K}),
\forall K\in\om_{h}\}$. It remains to construct the operator
$\vpidS$. We do so, based on a set of degrees of freedom given
in~\cite{GopalGuzma10a}, in the next lemma.

\begin{lemma} \label{projection_div_S} 
  There is an operator $\vpidS :
  H(\mathrm{div},K;\mathbb{S})\rightarrow P_{p+2}(K;\mathbb{S})$ such
  that for every $\tau \in H(\mathrm{div},K;\mathbb{S})$, we have
\begin{subequations}
  \label{eq:projection_div_S}
\begin{align}
\label{projection_div_S_1}
(\vpidS\tau,q_{p})_{K}
&  = (\tau,q_{p})_{K}, && \forall q_{p}\in P_{p}(K;\mathbb{S}),\\
\label{projection_div_S_2}
\langle \vpidS\tau\cdot n,\mu_{p+1}\rangle_{\partial K} 
& =
\ip{\mu_{p+1},\tau\cdot n}_{1/2, \d K},
&& \forall \mu_{p+1}\in\tilde{P}_{p+1}(\partial K;\mathbb{V}),\\
\label{projection_div_S_3}
\Vert\vpidS\tau\Vert_{\Hdiv{K}} &  \leq C\Vert\tau\Vert_{\Hdiv{K}}. 
\end{align}
\end{subequations} 
\end{lemma}

\begin{proof}
  We only give the proof for $N=3$ as the proof for $N=2$ is
  similar.  As in the proof of Lemma~\ref{lem:projection_div}, we will
  first construct the operator on the unit simplex $\hat{K}$ in
  $\mathbb{R}^{N}$.
Define $P_{p+1}^{\bot}(\partial \hat{K};\mathbb{V})
=L^{2}(\partial \hat{K};\VVV)$-orthogonal complement of $\tilde{P}_{p+1}(\partial \hat{K};\mathbb{V})$ 
in $P_{p+1}(\partial \hat{K};\mathbb{V})$ and set
$
P_{p+2}^{0}(\hat{K};\mathbb{S})=\{\hat{\tau}\in P_{p+2}(\hat{K};\mathbb{S}):\;
\langle \hat{s},\hat{\tau}\hat{n}^{-}\cdot \hat{n}^{+}\rangle_{\hat{e}}=0,\;
\forall\hat{s}\in P_{p+2}(\hat{e}),\;\forall \hat{e}\in\triangle_{1}(\hat{K})\}
$
where, for each edge $\hat{e}\in\triangle_{1}(\hat{K})$, $\hat{n}^{+}$
and $\hat{n}^{-}$ are the normal vectors of the two faces sharing
$\hat{e}$.  Let 
$
B_{p+2}^{\mathrm{div}}(\hat{K};\mathbb{S})=\{\hat{\tau}\in P_{p+2}^{0}(\hat{K};\mathbb{S}):
\langle\hat{v},\hat{\tau}\hat{n}\rangle_{\partial\hat{K}}=0$ for all 
$\hat{v}\in P_{p+1}^{\bot}(\partial\hat{K};\mathbb{V})\}.
$
We define $\hat{\varPi}_{p+2}^{(\mathrm{div},\mathbb{S})}
:\;H(\mathrm{div},\hat{K};\mathbb{S}) \mapsto
B_{p+2}^{\mathrm{div}}(\hat{K};\mathbb{S})$ by
\begin{subequations}
\begin{align}
\label{projection_div_S_reference_1}
(\hat{\varPi}_{p+2}^{(\mathrm{div},\mathbb{S})}\hat{\tau},\hat{q}_{p})_{\hat{K}} & = (\hat{\tau},\hat{q}_{p})_{\hat{K}}, &&\forall \hat{q}_{p}
\in P_{p}(\hat{K};\mathbb{S}),\\
\label{projection_div_S_reference_2}
\langle \hat{\varPi}_{p+2}^{(\mathrm{div},\mathbb{S})}\hat{\tau}\cdot \hat{n},\hat{\mu}_{p+1}\rangle_{\partial \hat{K}} & =
\ip{ \hat{\mu}_{p+1}, \hat{\tau}\cdot \hat{n}}_{1/2,\d K},
&&\forall \hat{\mu}_{p+1}\in\tilde{P}_{p+1}(\partial \hat{K};\mathbb{V}).
\end{align}
\end{subequations}
By~\cite[Theorem~2.1]{GopalGuzma10a}, these equations are uniquely
solvable, so $\hat{\varPi}_{p+2}^{(\mathrm{div},\mathbb{S})}$ is well
defined.

Next, we define $\vpidS$ on any general simplex $K$ by mapping
$\hat{\varPi}_{p+2}^{(\mathrm{div},\mathbb{S})}$ from $\hat{K}$ using
the Piola transform for symmetric matrix-valued functions. Recalling
the mapping $G_K$ from $\hat K$ onto $K$ and its derivative $A$, we
define
\begin{equation*}
\vpidS\tau (x) = \dfrac{1}{\det A}A \hat{\varPi}_{p+2}^{(\mathrm{div},\mathbb{S})}\hat{\tau}(\hat{x}) A^{t},
\end{equation*}
for any $\tau \in H(\mathrm{div},K;\mathbb{S})$.  Here, given $\tau$
on $K$, the function $\hat \tau$ on $\hat K$ is defined by $(\det
A)\tau (x) = A\hat{\tau}(\hat{x})A^{t},$ with $x=G_{K}(\hat{x}).$
As in the proof of Lemma~\ref{lem:projection_div}, it is now easy to
see that $\vpidS \tau$ satisfies~\eqref{projection_div_S_1} and ~\eqref{projection_div_S_2}.


Next, we observe that the commutativity property
\begin{equation}
\label{eq:7}
\dive \vpidS \tau = \varPi_{p+1}\dive\tau,
\end{equation}
holds, where $\vpi_{p+1}$ denotes the $L^2(K;\VVV)$-orthogonal 
projection onto $P_{p+1}(K,\VVV)$. Let
$\omega_{p+1}\in P_{p+1}(K;\mathbb{V})$. Then
\begin{align*}
(\dive(\vpidS\tau),\omega_{p+1})_{K} & =-(\vpidS\tau,\grad \omega_{p+1})_{K}
+\langle (\vpidS\tau)\cdot n,\omega_{p+1}\rangle_{\partial K}
\\
 & = -(\vpidS\tau,
\veps( \omega_{p+1}))_{K}
+\langle (\vpidS\tau)\cdot n,\omega_{p+1}\rangle_{\partial K}\\
 & = -(\tau,\veps(\omega_{p+1}))_{K}
+ \ip{\omega_{p+1}, \tau\cdot n}_{1/2,\d K},
\qquad(\text{by }
(\ref{projection_div_S})),\\
 & = -(\tau,\grad \omega_{p+1})_{K}
+ \ip{\omega_{p+1}, \tau\cdot n}_{1/2,\d K},
\\
 & = (\dive\tau,\omega_{p+1})_{K}.
\end{align*}
which proves~\eqref{eq:7}.

It only remains to prove the estimate
of~\eqref{projection_div_S_3}. This can now be done as in the proof of
the estimate~\eqref{projection_div_3} of
Lemma~\ref{lem:projection_div}, in two steps, using~\eqref{eq:7} in
place of the commutativity property used there.
\end{proof}

The main result of this section is the following.

\begin{theorem}
  \label{thm:elasticity}
  Suppose that $r\ge p+N$ and suppose that the compliance tensor
  $\mathcal{A}$ is element-wise constant. Then,
  the difference between the
  discrete solution of the practical DPG method, 
  $ \uuh =
  (\sigma_h,u_h,\hat{u}_h,\hat{\sigma}_{n,h},\alpha_h)$, 
  and the exact solution 
  $ \uu =
  (\sigma,u,\hat{u},\hat{\sigma}_{n},\alpha)$
  satisfies
  \begin{multline*}
    \| \sigma - \sigma _h \|_{L^2(\om)} + 
    \| u - u_h \|_{L^2(\om)} + 
    \| \hat{u} - \hat{u}_h \|_{H_0^{1/2}(\d\oh)} + 
    \| \hat\sigma_n - \hat{\sigma}_{n,h} \|_{H^{-1/2}(\d\oh)}+
    \vert \alpha -\alpha_h \vert
    \\
    \le C
    \inf_{ (\rho_h,w_h,\hat{w}_h,\hat{\eta}_{h}, 
      ) \in U_h }
    \bigg(
      \| \sigma - \rho_h \|_{L^2(\om)} + 
      \| u - w_h \|_{L^2(\om)} 
      \\ + 
      \| \hat{u} - \hat{u}_h \|_{H_0^{1/2}(\d\oh)} + 
      \| \hat\sigma_n - \hat{\sigma}_{n,h} \|_{H^{-1/2}(\d\oh)}
    \bigg).
  \end{multline*}
\end{theorem}
\begin{proof}
  As mentioned above, we only need to verify Assumption~\eqref{eq:pi}
  for the $\vpi$ in~\eqref{eq:Pi-S} and apply
  Theorem~\ref{thm:approach}. By the inequalities of previous lemmas,
  the estimate~\eqref{eq:pi-est} is obvious. To
  prove~\eqref{eq:pi-iden}, namely $b( \ww, \vv - \vpi \vv)=0$ for all
  $\ww \in U_h$, it suffices to prove the following eight identities

\begin{align*}
    (\mathcal{A}\rho_h, \tau - \vpidS \tau )_\om & =0, 
    &
    (w_h, \dive (\tau - \vpidS\tau))_\oh & =0,
    \\
    \ip{ \hat{w}_h, (\tau - \vpidS \tau)\cdot n}_{\d\oh}&=0,
    &
    (\rho_h, \grad (v - \vpig v) )_\oh & = 0,
    \\
    \ip{ v - \vpig v , \hat{\eta}_h}_{\d\oh} & = 0,
    &
    Q_{0}^{-1}(\gamma_h I, \mathcal{A}(\tau - \vpidS\tau))_\om & = 0
    \\
    (\rho_h, q - \varPi_{p}^\KKK q)_\om & = 0,
    &
    Q_{0}^{-1}(\mathcal{A} \rho_h, (\beta-\beta)I)_\om & = 0, 
\end{align*}
for all $\ww\equiv (\rho_h,w_h,\hat{w}_h,\hat{\eta}_{h},\gamma_h) \in
U_h.$ The first five are proved exactly as in the proof of
Theorem~\ref{thm:laplace} but using the new lemma. The sixth is
obvious from~\eqref{projection_div_S_1}. To see the seventh, denoting
by $\skw \rho_h$ the skew-symmetric part of $\rho_h$, observe that
$(\rho_h, q - \varPi_{p}^\KKK q)_\om = (\skw(\rho_h), q -
\varPi_{p}^\KKK q)_\om = 0$, by the definition of $\vpi_p^\KKK$. 

Therefore, applying Theorem~\ref{thm:approach}, we obtain a
quasioptimality estimate. This yields the estimate of the theorem
after observing that in the infimum over $\ww\equiv
(\rho_h,w_h,\hat{w}_h,\hat{\eta}_{h},\gamma_h)$ in $U_h$, we may
choose $\gamma_h = \alpha$.
\end{proof}

We conclude by noting that results similar to
Corollary~\ref{cor:rates} and Theorem~\ref{thm:cond} can be
established for this example as well. The arguments are very similar.

\end{document}